\documentclass[11pt,reqno]{amsart}

\usepackage{amssymb}
\usepackage{amsfonts}
\usepackage{amsthm}
\usepackage{mathrsfs}

\theoremstyle{plain} 
\newtheorem{theorem}{Theorem}

\newtheorem{lemma}[theorem]{Lemma}
\theoremstyle{definition} 
\newtheorem{definition}[theorem]{Definition}
\newtheorem{remark}[theorem]{Remark}
\newtheorem{example}[theorem]{Example}
\newcommand{\R}{\ensuremath{\mathbb{R}}}

\newcommand{\N}{\ensuremath{\mathbb{N}}}
\newcommand{\C}{\ensuremath{\mathbb{C}}}

\numberwithin{equation}{section}
\numberwithin{theorem}{section}

\small\normalsize

\begin{document}
\author[anderson]{Douglas R. Anderson}
\author[otto]{Jenna M. Otto}
\title[hyers--ulam stability of linear singular differential equations]{Hyers-Ulam stability of certain singular linear differential equations}
\address{Department of Mathematics, Concordia College, Moorhead, MN 56562 USA}
\email{andersod@cord.edu, jotto1@cord.edu}
\urladdr{http://www.cord.edu/faculty/andersod/}


\begin{abstract}
We establish the Hyers-Ulam stability of certain linear first-order differential equations with singularities. We then extend these results to higher-order singular linear differential equations that can be written with these first-order factors. An example of our results is given for a second-order singular linear differential equation that is not be covered by the current literature in this area.
\end{abstract}

\maketitle


\thispagestyle{empty}


\section{introduction}
In 1940, Ulam \cite{ulam} posed the following problem concerning the stability of functional equations: give conditions in order for a linear mapping near an approximately linear mapping to exist. The problem for the case of approximately additive mappings was solved by Hyers \cite{hyers} who proved that the Cauchy equation is stable in Banach spaces, and the result of Hyers was generalized by Rassias \cite{rassias}. Obloza \cite{ag} appears to be the first author who investigated the Hyers-Ulam stability of a differential equation.

Since then there has been a significant amount of interest in Hyers-Ulam stability, especially in relation to ordinary differential equations, for example see \cite{gavruta,gjl,jung,jung2,jung3,jung4,lishen1,lishen2,mmt,mmt2,pr,pr2,rus,wang}. Also of interest are many of the articles in a special issue guest edited by Rassias \cite{rass2}, dealing with Ulam, Hyers-Ulam, and Hyers-Ulam-Rassias stability in various contexts. Also see Popa et al \cite{bpx, popa, pr, pr2}.  Andr\'{a}s and M\'{e}sz\'{a}ros \cite{andras} recently used an operator approach to show the stability of linear dynamic equations on time scales with constant coefficients, as well as for certain integral equations.  Anderson et al \cite{and} considered the Hyers-Ulam stability of second-order linear dynamic equations with nonconstant coefficients, while Tun\c{c} and Bi\c{c}er \cite{tunc} proved the Hyers-Ulam stability of third and fourth-order Cauchy-Euler differential equations. 


\section{Hyers-Ulam stability for first-order equations}

In \cite[Theorem 2.2]{pr2}, Popa and Ra\c{s}a prove the Hyers-Ulam stability of
\begin{equation}\label{popaeq}
 y'(t)+\lambda(t)y(t)=f(t), \quad t\in I=(a,b), 
\end{equation}
where $a,b\in\R\cup\{\pm\infty\}$, assuming that the condition
\begin{equation}\label{popacondition}
 \inf_{t\in I}|\Re\lambda(t)|:=m>0 
\end{equation}
is met, where $\Re z$ is the real part of the complex number $z$. We will consider several singular differential equations that can be written in the form \eqref{popaeq}, where we do not assume \eqref{popacondition}. Note that Hyers-Ulam stability is independent of the nonhomogeneous term $f$ in \eqref{popaeq}, so in the sequel we consider only homogeneous equations, that is $f\equiv 0$.

Before we present the main results of this section, we recall the definition of Hyers-Ulam stability.


\begin{definition}[Hyers-Ulam stability]\label{def21}
Let $\varphi:(a,\infty)\rightarrow\R$ be a continuous function, $z\in\C$ be a complex constant, and $\varepsilon>0$. If whenever a differentiable function $x:(a,\infty)\rightarrow\C$ satisfies
$$ \left| \varphi(t) x'(t)+zx(t) \right|\le \varepsilon, \quad t\in(a,\infty) $$
there exists a solution $y$ of $\varphi(t) y'(t)+zy(t)=0$ such that $|y-x|\le K\varepsilon$ on $(a,\infty)$ for some constant $K>0$, then equation $\varphi(t) y'(t)+zy(t)=0$ has Hyers-Ulam stability $(a,\infty)$. Similar definitions hold on $(0,a)$ for $a>0$.
\end{definition}


\begin{lemma}\label{lem22}
Let $z\in\C$ and $\gamma\in\R$ be given constants. If $\Re z\ne 0$, or if $\Re z=0$ and $\gamma>1$, 
then the equation
\begin{equation}\label{djeq1}
 t^\gamma y'(t)+zy(t)=0, \quad t\in (1,\infty), 
\end{equation}
is Hyers-Ulam stable on $(1,\infty)$. If $\Re z=0$ and $\gamma\le 1$, then \eqref{djeq1} is unstable on $(1,\infty)$ in the Hyers-Ulam sense. The values $K>0$ for the Hyers-Ulam stability of \eqref{djeq1} are given in Table $\ref{tab:(1,infty)K}$.
\end{lemma}

\begin{table}[t]
\centering
\begin{tabular}{|c|ccc|} \hline
$(1,\infty)$ & $\Re z<0$ & $\Re z=0$ & $\Re z>0$ \\ \hline
$\gamma<1$ &  $\displaystyle\frac{1}{|\Re z|}$ & None & $\displaystyle\frac{1}{\Re z}$ \\
$\gamma=1$ & $\displaystyle\frac{1}{|\Re z|}$ & None & $\displaystyle\frac{1}{\Re z}$ \\
$\gamma>1$ & $\displaystyle\frac{1}{\Re z}\left(1-e^{\frac{\Re z}{1-\gamma}}\right)$ & $\displaystyle\frac{1}{\gamma-1}$ & $\displaystyle\frac{1}{\Re z}\left(1-e^{\frac{\Re z}{1-\gamma}}\right)$ \\ \hline
\end{tabular}
\caption{$K$ values for Hyers-Ulam stability of \eqref{djeq1} on $(1,\infty)$}
\label{tab:(1,infty)K}
\end{table}

\begin{proof}
Let $t_0\in(1,\infty)$. If $x:(1,\infty)\rightarrow\C$ is an approximate solution of \eqref{djeq1} such that
$$ t^\gamma x'(t)+zx(t)=q(t), \quad |q(t)| \le \varepsilon, \quad t\in(1,\infty), $$
for some perturbation $q:(1,\infty)\rightarrow\C$ and some $\varepsilon>0$, pick $y:(1,\infty)\rightarrow\C$ to be a solution of \eqref{djeq1} in the following way:
Let $A=\infty$ if $\Re z<0$ and $\gamma\le 1$, but let $A=1$ otherwise. If $\gamma=1$, then
\begin{equation}\label{yxc1}
 y(t)=ct^{-z}, \quad x(t)=y(t)+t^{-z}\int_{A}^{t}s^{z-1}q(s)ds, \quad c=t_0^zx(t_0)+\int_{t_0}^{A}s^{z-1}q(s)ds;
\end{equation}
if $\gamma\ne 1$, then
\begin{equation}\label{yxcnot1}
 y(t)=ce^{\frac{-zt^{1-\gamma}}{1-\gamma}}, \quad x(t)=y(t)+e^{\frac{-zt^{1-\gamma}}{1-\gamma}}\int_{A}^{t}s^{-\gamma}q(s)e^{\frac{zs^{1-\gamma}}{1-\gamma}}ds, \quad c=e^{\frac{zt_0^{1-\gamma}}{1-\gamma}}x(t_0)+\int_{t_0}^{A}s^{-\gamma}q(s)e^{\frac{zs^{1-\gamma}}{1-\gamma}}ds.
\end{equation}
We now proceed by cases to verify the $K$ values in Table \ref{tab:(1,infty)K}, where $K>0$ is the Hyers-Ulam constant in Definition \ref{def21}.

If $\Re z<0$ and $\gamma\le 1$, then $A=\infty$ and by \eqref{yxc1} and \eqref{yxcnot1} we have
\[ |y(t)-x(t)| = \begin{cases} 
\displaystyle\left|t^{-z}\int^{\infty}_{t}s^{z-1}q(s)ds\right| \le \varepsilon t^{-\Re z}\int^{\infty}_{t}s^{\Re z-1}ds = \varepsilon\frac{1}{|\Re z|} &: \gamma=1, \\ & \\
\displaystyle\left|e^{\frac{-zt^{1-\gamma}}{1-\gamma}}\int^{\infty}_{t}s^{-\gamma}q(s)e^{\frac{zs^{1-\gamma}}{1-\gamma}}ds\right| \le \varepsilon e^{\frac{-t^{1-\gamma}\Re z}{1-\gamma}}\int^{\infty}_{t}s^{-\gamma}e^{\frac{s^{1-\gamma}\Re z}{1-\gamma}}ds = \varepsilon\frac{1}{|\Re z|} &: \gamma<1.
\end{cases} \]

If $\Re z>0$ and $\gamma\le 1$, then $A=1$ and by \eqref{yxc1} and \eqref{yxcnot1} we have
\[ |y(t)-x(t)| = \begin{cases} 
\displaystyle\left|t^{-z}\int_{1}^{t}s^{z-1}q(s)ds\right| \le \varepsilon \frac{t^{-\Re z}}{\Re z}\left(t^{\Re z}-1\right) \le \varepsilon\frac{1}{\Re z} &: \gamma=1, \\ & \\
\displaystyle\left|e^{\frac{-zt^{1-\gamma}}{1-\gamma}}\int_{1}^{t}s^{-\gamma}q(s)e^{\frac{zs^{1-\gamma}}{1-\gamma}}ds\right| \le \frac{\varepsilon}{\Re z} \left(1-e^{\frac{\Re z}{1-\gamma}\left(1-t^{1-\gamma}\right)}\right) \le \varepsilon\frac{1}{\Re z} &: \gamma<1.
\end{cases} \]

If $\gamma>1$, then $A=1$ and by \eqref{yxc1} and \eqref{yxcnot1} we have
\[ |y(t)-x(t)| = \begin{cases}
\displaystyle\left|e^{\frac{-i\beta t^{1-\gamma}}{1-\gamma}}\int_{1}^{t}s^{-\gamma}q(s)e^{\frac{i\beta s^{1-\gamma}}{1-\gamma}}ds\right| \le \varepsilon \int_{1}^{t}s^{-\gamma}ds \le \varepsilon\frac{1}{\gamma-1} &: z=i\beta, \\ & \\
\displaystyle\left|e^{\frac{-zt^{1-\gamma}}{1-\gamma}}\int_{1}^{t}s^{-\gamma}q(s)e^{\frac{zs^{1-\gamma}}{1-\gamma}}ds\right| \le \frac{\varepsilon}{\Re z} \left(1-e^{\frac{\Re z}{1-\gamma}\left(1-t^{1-\gamma}\right)}\right) \le \frac{\varepsilon}{\Re z}\left(1-e^{\frac{\Re z}{1-\gamma}}\right) &: \Re z\ne 0.
\end{cases} \]

If $\Re z=0$ and $\gamma\le 1$, please see the proof of Theorem \ref{thm25} below for specific examples illustrating that \eqref{djeq1} is not Hyers-Ulam stable on $(1,\infty)$ for these values. All together, the result is proven and Table $\ref{tab:(1,infty)K}$ is verified.
\end{proof}


\begin{lemma}\label{lem23}
Let $z\in\C$ and $\gamma\in\R$ be given constants. If $\Re z\ne 0$, or if $\Re z=0$ and $\gamma<1$, 
then the singular equation
\begin{equation}\label{djeq2}
 t^\gamma y'(t)+zy(t)=0, \quad t\in (0,1), 
\end{equation}
is Hyers-Ulam stable on $(0,1)$. If $\Re z=0$ and $\gamma\ge 1$, then \eqref{djeq2} is unstable on $(0,1)$ in the Hyers-Ulam sense. The values $K>0$ for the Hyers-Ulam stability of \eqref{djeq2} are given in Table $\ref{tab:(0,1)K}$.
\end{lemma}

\begin{proof}
The proof proceeds in a way similar to that of the proof of Lemma \ref{lem22}.
Let $t_0\in(0,1)$. If $x:(0,1)\rightarrow\C$ is an approximate solution of \eqref{djeq2} such that
$$ t^\gamma x'(t)+zx(t)=q(t), \quad |q(t)| \le \varepsilon, \quad t\in(0,1), $$
for some $q:(0,1)\rightarrow\C$ and some $\varepsilon>0$, pick $y:(0,1)\rightarrow\C$ to be a solution of \eqref{djeq2} in the following way:
Let $A=0$ if $\Re z>0$ and $\gamma\ge 1$, but let $A=1$ otherwise. If $\gamma=1$, then $y$, $x$, and $c$ are given in \eqref{yxc1}; if $\gamma\ne 1$, then
$y$, $x$, and $c$ are given in \eqref{yxcnot1}. We again proceed by cases to verify the $K$ values in Table \ref{tab:(0,1)K}.

If $\Re z>0$ and $\gamma\ge 1$, then $A=0$ and by \eqref{yxc1} and \eqref{yxcnot1} we have
\[ |y(t)-x(t)| = \begin{cases} 
\displaystyle\left|t^{-z}\int_{0}^{t}s^{z-1}q(s)ds\right| \le \varepsilon t^{-\Re z}\int_{0}^{t}s^{\Re z-1}ds = \varepsilon\frac{1}{\Re z} &: \gamma=1, \\ & \\
\displaystyle\left|e^{\frac{-zt^{1-\gamma}}{1-\gamma}}\int_{0}^{t}s^{-\gamma}q(s)e^{\frac{zs^{1-\gamma}}{1-\gamma}}ds\right| \le \varepsilon e^{\frac{-t^{1-\gamma}\Re z}{1-\gamma}}\int_{0}^{t}s^{-\gamma}e^{\frac{s^{1-\gamma}\Re z}{1-\gamma}}ds = \varepsilon\frac{1}{\Re z} &: \gamma>1.
\end{cases} \]

If $\Re z<0$ and $\gamma\ge 1$, then $A=1$ and by \eqref{yxc1} and \eqref{yxcnot1} we have
\[ |y(t)-x(t)| = \begin{cases} 
\displaystyle\left|t^{-z}\int_{1}^{t}s^{z-1}q(s)ds\right| \le \varepsilon \frac{t^{-\Re z}}{|\Re z|}\left|t^{\Re z}-1\right| \le \varepsilon\frac{1}{|\Re z|} &: \gamma=1, \\ & \\
\displaystyle\left|e^{\frac{-zt^{1-\gamma}}{1-\gamma}}\int_{t}^{1}s^{-\gamma}q(s)e^{\frac{zs^{1-\gamma}}{1-\gamma}}ds\right| \le \frac{\varepsilon}{|\Re z|} \left|e^{\frac{-\Re z}{\gamma-1}\left(1-\frac{1}{t^{\gamma-1}}\right)}-1\right| \le \varepsilon\frac{1}{|\Re z|} &: \gamma>1.
\end{cases} \]

If $\gamma<1$, then $A=1$ and by \eqref{yxc1} and \eqref{yxcnot1} we have
\[ |y(t)-x(t)| = \begin{cases}
\displaystyle\left|e^{\frac{-i\beta t^{1-\gamma}}{1-\gamma}}\int_{t}^{1}s^{-\gamma}q(s)e^{\frac{i\beta s^{1-\gamma}}{1-\gamma}}ds\right| \le \varepsilon \int_{t}^{1}s^{-\gamma}ds \le \varepsilon\frac{1}{1-\gamma} &: z=i\beta, \\ & \\
\displaystyle\left|e^{\frac{-zt^{1-\gamma}}{1-\gamma}}\int_{t}^{1}s^{-\gamma}q(s)e^{\frac{zs^{1-\gamma}}{1-\gamma}}ds\right| \le \frac{\varepsilon}{|\Re z|} \left|e^{\frac{\Re z}{1-\gamma}\left(1-t^{1-\gamma}\right)}-1\right| \le \frac{\varepsilon}{\Re z}\left(e^{\frac{\Re z}{1-\gamma}}-1\right) &: \Re z\ne 0.
\end{cases} \]

If $\Re z=0$ and $\gamma\ge 1$, please see the proof of Theorem \ref{thm25} below for specific examples illustrating that \eqref{djeq2} is not Hyers-Ulam stable on $(0,1)$ for these values. All together, the result is proven and Table $\ref{tab:(0,1)K}$ is verified.
\end{proof}

\begin{table}[t]
\centering
\begin{tabular}{|c|ccc|} \hline
$(0,1)$    & $\Re z<0$ & $\Re z=0$ & $\Re z>0$ \\ \hline
$\gamma<1$ & $\displaystyle\frac{1}{\Re z}\left(e^{\frac{\Re z}{1-\gamma}}-1\right)$ & $\displaystyle\frac{1}{1-\gamma}$ & $\displaystyle\frac{1}{\Re z}\left(e^{\frac{\Re z}{1-\gamma}}-1\right)$ \\
$\gamma=1$ & $\displaystyle\frac{1}{|\Re z|}$ & None & $\displaystyle\frac{1}{\Re z}$ \\
$\gamma>1$ & $\displaystyle\frac{1}{|\Re z|}$ & None & $\displaystyle\frac{1}{\Re z}$ \\ \hline
\end{tabular}
\caption{$K$ values for Hyers-Ulam stability of \eqref{djeq2} on $(0,1)$}
\label{tab:(0,1)K}
\end{table}


\begin{remark}\label{renonhomo}
Consider the nonhomogeneous version of \eqref{djeq1}, namely
\begin{equation}\label{nonhomof}
  t^\gamma y'(t)+zy(t)=f(t), \quad t\in (1,\infty),
\end{equation}
for some continuous function $f:(1,\infty)\rightarrow\R$. We can easily modify the proof of Lemma \ref{lem22} as follows. If $x:(1,\infty)\rightarrow\C$ is an approximate solution of \eqref{nonhomof} such that
$$ t^\gamma x'(t)+zx(t)-f(t)=q(t), \quad |q(t)| \le \varepsilon, \quad t\in(1,\infty), $$
for some $q:(1,\infty)\rightarrow\C$ and some $\varepsilon>0$, pick $y:(1,\infty)\rightarrow\C$ to be a solution of \eqref{nonhomof} in the following way:
Let $A=\infty$ if $\Re z<0$ and $\gamma\le 1$, but let $A=1$ otherwise. If $\gamma=1$, then
\begin{eqnarray}\label{yxc1ff}
\nonumber y(t) &=& ct^{-z}+t^{-z}\int_{A}^{t}s^{z-1}f(s)ds, \quad x(t)=y(t)+t^{-z}\int_{A}^{t}s^{z-1}q(s)ds, \\
            c  &=& t_0^zx(t_0)+\int_{t_0}^{A}s^{z-1}(f(s)+q(s))ds;
\end{eqnarray}
if $\gamma\ne 1$, then
\begin{eqnarray}\label{yxcnot1ff}
\nonumber y(t) &=& ce^{\frac{-zt^{1-\gamma}}{1-\gamma}}+e^{\frac{-zt^{1-\gamma}}{1-\gamma}}\int_{A}^{t}s^{-\gamma}f(s)e^{\frac{zs^{1-\gamma}}{1-\gamma}}ds, \quad x(t)=y(t)+e^{\frac{-zt^{1-\gamma}}{1-\gamma}}\int_{A}^{t}s^{-\gamma}q(s)e^{\frac{zs^{1-\gamma}}{1-\gamma}}ds, \\    
   c &=& e^{\frac{zt_0^{1-\gamma}}{1-\gamma}}x(t_0)+\int_{t_0}^{A}s^{-\gamma}(f(s)+q(s))e^{\frac{zs^{1-\gamma}}{1-\gamma}}ds.
\end{eqnarray}
As the key calculations are based on $|y(t)-x(t)|$, we see from \eqref{yxc1ff} and \eqref{yxcnot1ff} that nothing is changed due to $f$, as its terms subtract off. Lemma \ref{lem23} can likewise accommodate a nonhomogeneous term without effect.
\end{remark}


\begin{remark}
Combining the results on $(1,\infty)$ from Lemma \ref{lem22} with those on $(0,1)$ from Lemma \ref{lem23}, we have the following theorem on the half line $(0,\infty)$. Of course, the results above could just as easily be on $(a,\infty)$ and $(0,a)$ for any $a>0$.
\end{remark}


\begin{theorem}\label{thm25}
Let $z\in\C$ and $\gamma\in\R$ be given constants. The singular differential equation
\begin{equation}\label{djeq4}
 t^\gamma y'(t)+zy(t)=0, \quad t\in (0,\infty), 
\end{equation}
is Hyers-Ulam stable on $(0,\infty)$ if and only if $\Re z\ne 0$.
\end{theorem}

\begin{proof}
If $\Re z\ne 0$, the Hyers-Ulam stability of \eqref{djeq4} follows from Lemmas \ref{lem22} and \ref{lem23}, respectively. If $\Re z=0$, let $\gamma,\beta\in\R$ and $i=\sqrt{-1}$, and consider the singular differential equation
\begin{equation}\label{sharp1}
  t^\gamma y'(t)+i\beta y(t)=0, \quad t\in(0,\infty).
\end{equation}
We will show \eqref{sharp1} is unstable in the sense of Hyers and Ulam.

Given any $\varepsilon>0$, for $t\in(0,\infty)$ let 
\[ x(t) = \begin{cases} \varepsilon t^{-i\beta}\ln t &: \gamma=1, \\ \varepsilon \left(\frac{t^{1-\gamma}}{1-\gamma}\right)e^\frac{i\beta t^{1-\gamma}}{-1+\gamma} &: \gamma\ne 1. \end{cases} \] 
Then
\[ \left|t^\gamma  x'(t)+i\beta x(t)\right| = \begin{cases} \left|\varepsilon t^{-i\beta}\right|= \varepsilon &: \gamma=1, \\ \left|\varepsilon  e^\frac{i\beta t^{1-\gamma}}{-1+\gamma}\right|= \varepsilon &: \gamma\ne 1,  \end{cases} \] 
for all $t\in(0,\infty)$. Clearly 
\[ y(t) = \begin{cases} ct^{-i\beta} &: \gamma=1, \\ c e^\frac{i\beta t^{1-\gamma}}{-1+\gamma} &: \gamma\ne 1, \end{cases} \] 
where $c$ is a constant, is the only type of solution of \eqref{sharp1}, but 
\[ |y(t)-x(t)| = \begin{cases} |ct^{-i\beta}-\varepsilon t^{-i\beta}\ln t| = |c-\varepsilon\ln t| &: \gamma=1, \\ \left|c e^\frac{i\beta t^{1-\gamma}}{-1+\gamma}-\varepsilon\left(\frac{t^{1-\gamma}}{1-\gamma}\right)e^\frac{i\beta t^{1-\gamma}}{-1+\gamma}\right| = \left|c-\varepsilon\left(\frac{t^{1-\gamma}}{1-\gamma}\right)\right| &: \gamma\ne 1, \end{cases} \] 
is unbounded on $(0,\infty)$, for any choice of $c$. Consequently, \eqref{sharp1} is unstable in the sense of Hyers and Ulam on $(0,\infty)$ for any $\gamma\in\R$.
\end{proof}


\begin{remark}
Theorem \ref{thm25} highlights a difference between stability of equilibria and Hyers-Ulam stability for equation \eqref{djeq4} on $(0,\infty)$. For example, when $\gamma=z=1$, equation \eqref{djeq4} is Hyers-Ulam stable on $(0,\infty)$, but the trivial solution $y\equiv 0$ is unstable, since given any $\varepsilon>0$ and any initial point $t_0\in(0,\infty)$, $y_0(t)=\varepsilon  t_0/t$ is also a solution and 
\[ \lim_{t\rightarrow 0^+}|y_0(t)-y(t)|=\infty. \]
\end{remark}


\begin{theorem}\label{thm2.6}
Let $z\in\C$ and $\gamma\in\R$ be given constants. The singular differential equation
\begin{equation}\label{djeq3}
 t(\ln t)^\gamma y'(t)+zy(t)=0, \quad t\in (1,\infty), 
\end{equation}
is Hyers-Ulam stable on $(1,\infty)$ if and only if $\Re z\ne 0$.
\end{theorem}

\begin{proof}
Let $y:(0,\infty)\rightarrow\C$ be a solution of \eqref{djeq4}, and let $Y:(1,\infty)\rightarrow\C$ be given by $Y(t):=y(\ln t)$. Then
\[ t(\ln t)^\gamma Y'(t)+zY(t)=(\ln t)^\gamma y'(\ln t)+zy(\ln t)=0 \]
for $t\in(1,\infty)$, so that $Y$ is solution of \eqref{djeq3}. Similarly, if $Y:(1,\infty)\rightarrow\C$ is a solution of \eqref{djeq3}, then $y:(0,\infty)\rightarrow\C$ is a solution of \eqref{djeq4} via $y(t):=Y(e^t)$ with $t=\ln u$ for $u\in(1,\infty)$. Thus there is a one-to-one correspondence between the solutions of \eqref{djeq4} and the solutions of \eqref{djeq3}, and likewise between approximate solutions of \eqref{djeq4} and  \eqref{djeq3}, respectively. By Theorem \ref{thm25}, the singular equation \eqref{djeq4} is Hyers-Ulam stable on $(0,\infty)$ if and only if $\Re z\ne 0$. The result for \eqref{djeq3} follows.
\end{proof}

%
%

\section{factoring}

In this section we show how the results in the previous section can be incorporated into an investigation of Hyers-Ulam stability for certain higher-order singular linear differential equations with nonconstant coefficients.

Let $D$ be the differential operator defined by $Dy=y'$ for differentiable functions $y:(0,\infty)\rightarrow\C$, and $I$ the identity operator given by $Iy=y$. 
For a given function $\varphi:(0,\infty)\rightarrow\R$, let $(\varphi D)^0y=Iy=y$, $(\varphi D)y=\varphi y'$, and for positive integers $n$, let $(\varphi D)^ny=(\varphi D)^{n-1}\varphi Dy$.
We consider the higher-order singular linear differential equation
\begin{equation}\label{maineq}
 \sum_{k=0}^n \alpha_{n-k}(t^\gamma D)^k y(t)=0 
\end{equation}
for some real constants $\gamma$ and $\alpha_m$, where $\alpha_0=1$ for convenience. Note that if $\gamma=0$, then this is the $n$th-order linear constant coefficient differential equation
$$ \sum_{k=0}^n \alpha_{n-k} y^{(k)}(t)=0, $$
while if $\gamma=1$, this is a nested form of the $n$th-order linear Cauchy-Euler differential equation
$$ \sum_{k=0}^n \alpha_{n-k}(tD)^k y(t)=0. $$


\begin{remark}\label{remarkrho}
Our first task will be to factor \eqref{maineq} for the analysis to follow. To accomplish the factorization of \eqref{maineq}, we use the substitutions (see also \cite[Remark 2.2]{and2})
\[ \alpha_1 = \sum_i z_i, \quad \alpha_2 = \sum_{i<j} z_iz_j, \quad \alpha_3 = \sum_{i<j<k} z_iz_jz_k, \quad \alpha_4 = \sum_{i<j<k<\ell} z_iz_jz_kz_{\ell}, \]  
\[ \cdots \quad \alpha_m = \sum_{i_1<i_2<\cdots<i_m} z_{i_1}z_{i_2}\cdots  z_{i_m}, \quad \cdots \quad
 \alpha_n = z_{1}z_{2}z_{3}\cdots  z_{n}, \]
where $z_i\in\C$ for $i=1,2,\cdots,n$. Then we have the factorization of \eqref{maineq} given by
\begin{equation}\label{factorize}
 \sum_{k=0}^n \alpha_{n-k}(t^\gamma D)^k y(t)=\prod_{k=1}^{n} \left(t^\gamma D+z_kI\right)y(t)=0, \qquad n\in\N;
\end{equation}
see \cite[Lemma 2.3]{and2} and \cite[Section 2]{tunc} for more on this type of substitution and factorization. The following result for \eqref{maineq} is the main result in this section.
\end{remark}


\begin{theorem}[Hyers-Ulam Stability]\label{thmHUS}
For positive integer $n$, consider the higher-order singular differential equation \eqref{maineq} with real constants $\gamma$ and $\alpha_m$ for $m=0,1,\cdots, n$, where $\alpha_0=1$. Let the substitutions $z_1,\cdots,z_n$ be as given in Remark $\ref{remarkrho}$. Then the singular equation \eqref{maineq} has Hyers-Ulam stability on $(0,\infty)$ if and only if $\Re z_k\ne 0$ for each $k=1,2,\cdots,n$.
\end{theorem}

\begin{proof}
By Remark \ref{remarkrho} and \eqref{factorize} we have that \eqref{maineq} can be written in factored form as
\[ \sum_{k=0}^n \alpha_{n-k}(t^\gamma D)^k y(t) = \prod_{k=1}^{n} \left(t^\gamma D+z_kI\right)y(t)=0. \]
Now suppose there exists a function $x$ such that 
\begin{equation}\label{xprodep}
 \left| \prod_{k=1}^{n} \left(t^\gamma D+z_kI\right)x(t) \right|\le \varepsilon
\end{equation}
for some $\varepsilon>0$, for all $t\in(0,\infty)$. Define the new functions $y_n\equiv 0$, $x_0:=x$ and
\begin{equation}\label{eq33}
 x_k:=\left(t^\gamma D+z_kI\right)x_{k-1}, \quad k=1,\ldots,n.
\end{equation}
Then 
\[ x_k(t) = t^\gamma x'_{k-1}(t)+z_kx_{k-1}(t), \quad k=1,\ldots,n, \quad t\in(0,\infty), \]
and by \eqref{xprodep} and \eqref{eq33} (recall $y_n=0$ and $x_0=x$), we have
\begin{equation}\label{xnep}
 |y_n(t)-x_n(t)|=|x_n(t)| \le \varepsilon, \quad t\in(0,\infty) 
\end{equation}
by construction. Hyers-Ulam stability of \eqref{eq33} on $(0,\infty)$ with $k=n$ implies there exists a solution $y_{n-1}$ of
\[ \left(t^\gamma D+z_nI\right)y_{n-1}(t)=y_n(t) \] 
such that $|y_{n-1}(t)-x_{n-1}(t)| \le K_n \varepsilon$; substituting from \eqref{eq33} with $k=n-1$ this inequality becomes
\[ |y_{n-1}(t)-t^\gamma x'_{n-2}(t)-z_{n-1}x_{n-2}(t)| \le K_n \varepsilon, \]
and so on. We proceed by iterating the proofs of Lemmas \ref{lem22} and \ref{lem23}, and using Remark \ref{renonhomo}. Let $A_k=0$ if $\Re z_k>0$ and $\gamma\ge 1$; let $A_k=\infty$ if $\Re z_k<0$ and $\gamma\le 1$; let $A_k=1$ otherwise, for $k=1,\cdots,n$. Let $t_k\in(0,\infty)$ for $k=1,\cdots,n$. Let $x_{k-1}$ solve \eqref{eq33} with initial condition at $t_k$, and $y_{k-1}$ solve
\begin{equation}\label{yksol}
 \left(t^\gamma D+z_kI\right)y_{k-1} = y_k, \quad k=1,\ldots,n. 
\end{equation}
If $\gamma=1$, then
\begin{eqnarray}\label{yxc1fkk}
\nonumber y_{k-1}(t) &=& c_{k-1}t^{-z_k}+t^{-z_k}\int_{A_k}^{t}s^{z_k-1}y_k(s)ds,  \\ 
\nonumber x_{k-1}(t) &=& y_{k-1}(t)+t^{-z_k}\int_{A_k}^{t}s^{z_k-1}x_k(s)ds, \\
          c_{k-1}    &=& t_k^{z_k}x_{k-1}(t_k)+\int_{t_k}^{A_k}s^{z_k-1}(y_k(s)+x_k(s))ds;
\end{eqnarray}
if $\gamma\ne 1$, then
\begin{eqnarray}\label{yxcnot1kk}
\nonumber y_{k-1}(t) &=& c_{k-1}e^{\frac{-z_kt^{1-\gamma}}{1-\gamma}}+e^{\frac{-z_kt^{1-\gamma}}{1-\gamma}}\int_{A_k}^{t}s^{-\gamma}y_k(s)e^{\frac{z_ks^{1-\gamma}}{1-\gamma}}ds, \\ \nonumber x_{k-1}(t) &=& y_{k-1}(t)+e^{\frac{-z_kt^{1-\gamma}}{1-\gamma}}\int_{A_k}^{t}s^{-\gamma}x_k(s)e^{\frac{z_ks^{1-\gamma}}{1-\gamma}}ds, \\    
   c_{k-1} &=& e^{\frac{z_kt_k^{1-\gamma}}{1-\gamma}}x_{k-1}(t_k)+\int_{t_k}^{A_k}s^{-\gamma}(y_k(s)+x_k(s))e^{\frac{z_ks^{1-\gamma}}{1-\gamma}}ds.
\end{eqnarray}
Using \eqref{eq33}, \eqref{xnep}, and Tables \ref{tab:(1,infty)K} and \ref{tab:(0,1)K}, we have on $(0,\infty)$ that
\[ |y_{k-1}(t)-x_{k-1}(t)| \le \varepsilon \prod_{j=k}^{n}K_j, \quad K_j=
\begin{cases} 
 \displaystyle \frac{1}{|\Re z_j|} &: \Re z_j<0, \gamma\le 1, \\
 \displaystyle \frac{1}{\Re z_j} &: \Re z_j>0, \gamma\ge 1, \\
 \max\left\{\displaystyle\frac{1}{\Re z_j},\displaystyle\frac{1}{\Re z_j}\left(e^{\frac{\Re z_j}{1-\gamma}}-1\right)\right\} &: \Re z_j>0, \gamma<1, \\
 \max\left\{\displaystyle\frac{1}{|\Re z_j|},\displaystyle\frac{1}{|\Re z_j|}\left(e^{\frac{\Re z_j}{1-\gamma}}-1\right)\right\} &: \Re z_j<0, \gamma>1, \\
\end{cases} \]
starting with $k=n$ and proceeding down to $k=1$. In particular (recall $x_0=x$, the original approximate solution), we arrive at
\[ |y_{0}(t)-x_{0}(t)| = |y_{0}(t)-x(t)| \le \varepsilon \prod_{j=1}^{n}K_j, \quad t\in(0,\infty). \]
From \eqref{yksol}, and using the fact that $y_n=0$, we see that $y_0$ is a solution of the factored form of \eqref{maineq}. Thus \eqref{maineq} has Hyers-Ulam stability on $(0,\infty)$.
\end{proof}


\begin{example}
Consider the singular differential equation \eqref{factorize} on $(0,\infty)$ with $n=2$, $\gamma=2$, and $z_k=-k<0$ for $k=1,2$, that is
\[ t^2\left(t^2y'\right)'(t)-3t^2y'(t)+2y(t) = \left(t^2 D-1I\right)\left(t^2 D-2I\right)y(t)=0, \quad t\in(0,\infty). \]
Note that for each of the first-order factors,
\[ \inf_{t\in (0,\infty)}|\Re\lambda_k(t)|=\inf_{t\in (0,\infty)}\frac{-z_k}{t^2}=0 \]
thus violating \eqref{popacondition}.

Suppose there exists a differentiable function $x:(0,\infty)\rightarrow\C$ and a constant $\varepsilon>0$ such that
\[ \left|\left(t^2 D-1I\right)\left(t^2 D-2I\right)x(t)\right|\le\varepsilon. \]
As in the proof of Theorem \ref{thmHUS}, we use $x$ to construct a solution $y$ of \eqref{factorize} given by
\[ y(t)=e^{2-2/t}x(1)+\left(e^{2-2/t}-e^{1-1/t}\right)\left(x'(1)-2x(1)\right). \]
Then $y$ exists on $(0,\infty)$ with finite limits at the endpoints of the interval, and 
\[ |y(t)-x(t)|\le \frac{(e-1)(e^2-1)}{2}\varepsilon \] 
for all $t\in(0,\infty)$.
\end{example}

\section*{Acknowledgements}
This research was supported by an NSF STEP grant (DUE 0969568) and by a gift to the Division of Sciences and Mathematics at Concordia College made in support of undergraduate research.


\end{document}